\documentclass[letterpaper]{amsart}
\usepackage{amsmath}
\usepackage{amssymb}
\usepackage{amsthm}
\usepackage{color}

\newtheorem{prevtheorem}{Theorem}

\newtheorem{theorem}{Theorem}
\newtheorem{lemma}[theorem]{Lemma}

\theoremstyle{definition}

\theoremstyle{remark}

\newcommand{\F}{\textbf{F}}

\title{Word Maps in Finite Simple Groups}
\date{\today \hspace{1mm} This material is based upon work done while the first author was supported by the National Science Foundation under Grant No. DMS-1502553.}
\author{William Cocke}
\address{William Cocke\\Department of Mathematics, University of Wisconsin-Madison}
\email{cocke@math.wisc.edu}

\author{Meng-Che ``Turbo'' Ho}
\address{Meng-Che ``Turbo'' Ho\\Department of Mathematics, Purdue University}
\email{ho140@purdue.edu}

\begin{document}

\begin{abstract}
Elements of the free group define interesting maps, known as word maps, on groups. It was previously observed by Lubotzky that every subset of a finite simple group that is closed under endomorphisms occurs as the image of some word map. We improve upon this result by showing that the word in question can be chosen to be in any $v(\textbf{F}_n)$ provided that $v$ is not a law on the finite simple group in question. In addition, we provide an example of a word $w$ that witnesses the chirality of the Mathieu group $M_{11}$. The paper concludes by demonstrating that not every subset of a group closed under endomorphisms occurs as the image of a word map.  
\end{abstract}

\maketitle

\section{Introduction}

The image of various word maps in finite simple groups has been a topic of considerable interest. Most famously the now-proved Ore conjecture asked whether every element of a finite nonabelian simple group $G$ is a commutator \cite{LOST}. Recently, for any finite nonabelian simple group $G$, it was shown that if $N$ is the product of two prime powers, then every element of $G$ occurs as the product of two $N$-powers in $G$ \cite{GLOST}. 

A word $w$ is an element of the free group $w\in \textbf{F}_n=\textbf{F}\langle x_1,\dots,x_n \rangle$. For any group $G$, the word $w$ induces a map $w:G^n\rightarrow G$, where \[ (g_1,\dots,g_n) \rightarrow w(g_1,\dots,g_n).\] We write $w$ for both the word $w$ and the word map on $G$, and write $w(G)$ to mean the image $w(G^n)$ of the word map. In general $w(G)$ is not a group, but merely a subset of $G$. We write $\langle w(G) \rangle$ for the group generated by $w(G)$. We will also write $\overline{g}$ to mean the tuple $(g_1,g_2,\cdots,g_n)$. In the notation now established, the Ore conjecture asked whether for every finite simple group $G$, the word $w=[x,y]=x^{-1} y^{-1} x y$ satisfies $w(G)=G$. 

Word maps represent an interesting class of functions on groups. In general, word maps are not homomorphisms, but they do respect automorphisms and endomorphisms of groups. Explicitly we have the following lemma.

\begin{lemma}
Let $w \in \textbf{F}_n$ and $G$ be a group. Then for any $\overline{g}\in G$ and $\varphi:G\rightarrow G$ the following holds. 
\[
\varphi\left(w(\overline{g})\right) = w\left(\overline{\varphi(g)} \right).
\]
\end{lemma}

Hence $w(G)$ is closed under all endomorphisms from $G$ to $G$. In examining word maps on finite simple groups, the question was asked at the conference `Words and Growth' (Jerusalem, June 2012) if every subset of a finite simple group that is closed under endomorphisms of $G$ occurs as the image of some word map. Lubotzky responded in the affirmative with the following theorem.

\begin{theorem}\cite{Lubotzky}
Let $G$ be a finite simple group, $n>1$, and let $A\subseteq G$ such that $A$ is closed under all endomorphisms of $G$, then there is a word $w\in \textbf{F}_n$ such that $A=w(G)$. 
\end{theorem}

In the current work we do two things: we extend Lubotzky's result by showing that the structure of $w$ realizing $A$ can be controlled in a very strong way; we also show that there are groups $G$ and $A\subset G$ with $A$ closed under endomorphisms such that $A$ is not $w(G)$ for any $w$. As part of our generalization of Lubotzky's result, we provide a word $w$ so that over $G=M_{11}$, we have that $w(G)$ is exactly one of the conjugacy classes with representative of order $11$. The explicit realization of such a word provides a quick proof of observations by Gordeev et.\ al.\ \cite[7.3-7.4]{Gordeev}.

\begin{prevtheorem}\label{main}
Let $G$ be a finite simple group, $n>1$, and $A\subseteq G$ such that $A$ is closed under automorphisms and $1\in A$. Assume that $v\in \textbf{F}_n$ is not a law on $G$. Then there is a word $w\in  \langle v(\textbf{F}_n)\rangle$ such that $A=w(G)$. 
\end{prevtheorem}

We note Theorem \ref{main} shows that any subset of $G$ that is closed under endomorphisms of $G$ can occurs as the image of a word map $w$ in $v(\textbf{F}_n)$, it does not provide a description of $w$. However, it is possible in some cases to explicitly find $w$. We will show the following theorem which relates to the authors' earlier work on the chirality of groups \cite{CH}.

The Mathieu Group $M_{11}$ has two conjugacy classes of order 11 that are the inverse of each other. We construct a word whose image picks out exactly one of the conjugacy classes. Furthermore, although the word is long in word length, it is short as a straight-line program.

\begin{prevtheorem}\label{example}
Let $G$ be the Mathieu Group $M_{11}$ and let w be the word \[ [x^{-440}(x^{-440})^{(y^{-440})}x^{-440},(y^{-440})^{(x^{-440}y^{-440})}y^{-440}].\] Then $w(G)$ contains an element $g$ such that $o(g)=11$ and $g^{-1} \notin w(G)$, i.e., the word $w$ witnesses the chirality of $G$. 
\end{prevtheorem}

In the case of a general group $G$, one might ask if being closed under endomorphisms of $G$ is a sufficient condition for a subset $A$ to be $w(G)$ for some $G$. We will show this is false in Section \ref{abelian}, even in the case of abelian groups. 

\begin{prevtheorem}\label{neg}
Let $G$ be the cyclic group of order $12$. Then 
\[ A=\{x^2: x \in G\} \cup \{ x^3: x \in G\},\] is closed under endomorphisms of $G$, but is not the image of any word map over $G$. 
\end{prevtheorem}

\section{Proof of Theorems \ref{main} and \ref{example}}
Our proof of Theorem \ref{main} will rely on work done on the varieties of groups. Recall that a variety $\mathcal{B}$ is the class of all groups satisfying some set of laws $X$, i.e.\ a group $G$ is in $\mathcal{B}$ if and only if for every word $w\in X$ and every $n$-tuple $\overline{g}\in G$ we have that $w(\overline{g}) =1.$ For example, the variety of abelian groups is defined by the law $w=x^{-1} y^{-1} x y.$  Similarly, solvable groups with derived length $n$ or nilpotent groups of class $m$ are varieties. 

For a finite group $G$ there are only finitely many word maps on $G$. Moreover, if $w$ and $v$ are two word maps from $G^{n}\rightarrow G$, then $w\cdot v$ is a word map from $G^{n}\rightarrow G$ given by $(w\cdot v) (\overline{g}) = w(\overline{g})  v(\overline{g})$. Hence the set of word maps on $n$ variables over $G$ forms a group $\textbf{F}_n(G)$. We can equivalently define $\textbf{F}_n(G)$ as follows.

Let $K(G)$ be the set of all $n$-variable laws on $G$. Then $K(G)$ is a characteristic subgroup of $\textbf{F}_n$ and $\textbf{F}_n(G)= \textbf{F}_n/K(G)$. 

The group $\textbf{F}_n(G)$ is the free group of rank $n$ in the variety generated by $G$. In particular, any $n$-generated group in the variety generated by $G$ occurs as a quotient of $\textbf{F}_n(G)$. In H.\ Neumann's text \emph{Varieties of Groups}, it is observed that for a finite simple group $G$, we have 
\[
\textbf{F}_n(G) = G^{d(n)} \times \textbf{F}_n(H)
\]
where $H$ is the direct product of all proper subgroups of $G$ \cite[pg 141]{Ne67} and $d(n)$ is the number of orbits of $\operatorname{Aut}(G)$ acting on the generating $n$-tuples of $G$. However, since this occurs without proof, we will prove a slightly weaker statement below, which will be sufficient for our purposes. It is also the case that $G^{d(n)}$ is $n$-generated, but $G^{d(n)+1}$ is not \cite{Hall}. 

\begin{lemma} \label{product_lemma}
Let $G$ be a finite simple group. Then 
\[
\textbf{F}_n(G)=G^{d(n)} \times H,
\]
for some group $H$.
\end{lemma}

\begin{proof}
Since a word map $w$ respects endomorphisms of $G$, the map $w$ is defined by its value  on a set of representatives of the diagonal action of the automorphism groups of $G$ on $G^{n}$. Moreover the number of possible values of $w$ on an orbit representative $(\overline{g})$ is less than or equal to $|\langle \overline{g} \rangle |$, the size of the subgroup generated by the orbit.  There are exactly $d(g)$ orbits of $n$-tuples corresponding to $n$-tuples that generate $G$, and some number of other orbits. 

Therefore $\textbf{F}_n(G)$ is a subgroup of the direct product $G^{d(n)} \times K$ where $K$ is some direct product of proper subgroups of $G$. But, any group of rank $n$ that satisfies the same laws as $G$ occurs as a quotient of $\textbf{F}_n(G)$. Hence $G^{d(n)}$ must occur as a quotient of $\textbf{F}_n(G)$. 
\end{proof}

Before proving Theorem \ref{main}, we need the following lemma which follows from the work of Kantor and Guralnick, which depends heavily on the classification of finite simple groups \cite[Corollary p. 745]{KG}.

\begin{lemma}\cite{KG}
For every nontrivial element $g$ of a finite simple group $G$ there is an $h\in G$ such that $G=\langle g, h \rangle.$
\end{lemma} In particular, it is the case that for any finite simple group $G$, the number $d(n)$ is greater than the number of conjugacy classes of $G$. 

\begin{proof}[\textbf{Proof of Theorem \ref{main}.}]
Since $v$ is not a law on $G$, we know that $\langle v(G)\rangle =G$. By Lemma \ref{product_lemma} we see that \[\langle v\left(\textbf{F}_n(G)\right) \rangle=\langle v(G^{d(n)})\rangle \times \langle v(H) \rangle = G^{d(n)} \times \langle v(H)\rangle ,\] for the appropriate group $H$. 

Hence there is a word map $w\in v(\left(\textbf{F}_n(G)\right)$ that is defined by its value on the generating tuples with a value from the group $G^{d(n)} \times v(H)$. Given $A$, we can write $A$ as a union of $m\leq d(n)$ automorphism classes. There is a $w$ so that on $m$ different orbits of generating tuples of $G$, the value of $w$ is one of the distinct automorphism classes in $A$ and $w$ vanishes elsewhere. 

Now we need to find a word in $\langle v(\F_n) \rangle$ such that it induces the word map $w$ on $G$. Consider the word maps induced by the words $x_1,\cdots,x_n$. These word maps generate $\F_n(G)$. Since $w\in \langle v(\F_n(G))\rangle$, we can write $w$ as product of elements of the form $v(u_1,\cdots,u_n)$ such that each $u_j$ is a product of $x_1,\cdots,x_n$. Consider this spelling of $w$ in $x_1,\cdots,x_n$ as an element in $\F_n = \F(x_1,\cdots,x_n)$, we get a word in $\langle v(\F_n)\rangle$ that induces the word map $w$ on $G$, which has image being $A$.
\end{proof}

We will find a commutator word below that realizes an interesting property of the Mathieu group of order 7920. 

In \cite{CH2} the authors were interested in finding chiral word maps $w$, i.e.\ for some group $G$ we have that $g\in w(G)$ but $g^{-1} \notin w(G)$. Gordeev et.\ al.\ call the pair $(G,w)$ a chiral pair \cite{Gordeev}. They note that for groups known to be chiral, it is not always easy to produce a chiral pair, or witness of the chirality. For example, the Mathieu group of order 7920, $G=M_{11}$, is chiral as a result of Lubotzky's theorem or Theorem \ref{main}, but the only bound one has on the length of a $w$ in a chiral pair $(G,w)$ is the number of word maps over $M_{11}$, somewhere around $1.7 \times 10^{244552995}$. Below, we provide a word $w$ that has length 9680 and witnesses the chirality of $G$.

We will now show that the word \[w=[x^{-440}(x^{-440})^{(y^{-440})}x^{-440},(y^{-440})^{(x^{-440}y^{-440})}y^{-440}]\] witnesses the chirality of the Mathieu Group $M_{11}$. 

\begin{proof}[\textbf{Proof of Theorem \ref{example}}.]
All elements of $M_{11}$ have order either $1,2,3,4,5,6,8,$ or $11$. For an element $g$ of $M_{11}$ we have 
\[
g^{-440} =
\begin{cases}
1 & \text{ if } o(g) \notin \{3,6\}, \\
g & \text{ if } o(g) = 3, \\
g^4 & \text{ if } o(g) = 6.
\end{cases}
\]
If $a\in G$ does not have order $3$ or $6$, then $w(a,b)=w(1,b)=1$ for all $a\in G$, since $w$ is a commutator. Similarly, if $b\in G$ does not have order $3$ or $6$, then $w(a,b)=w(a,1)=1$ for all $a\in G$.  Moreover, $w(a,b)=w(a^4,b^4)$ for all $a,b \in G$. Hence to determine $w(G)$ we need to determine $w(a,b)$ where both $a$ and $b$ have order 3. There are 93600 such tuples from $G$. Let $X=\{a \in G : o(a)=3.\}$

Let $v=[x(x^y)x,y^{(xy)}y]$. Using Magma it is easy to compute the value of $w$ on all $a,b \in X$, by noting that $w(a,b)=v(a,b)$ \cite{Magma}. Computing the value of $w(a,b)$ for all $a,b\in X$, there are elements of order 1,2,4,5,6, and 11. However, all of the elements of order 11 that occur in the image of $w$ are conjugate.  For $g\in M_{11}$ with $o(g)=11$ we have that $g^{-1}\notin g^G$. We conclude that $w$ witnesses the chirality of $M_{11}$. 
\end{proof}

We note that $w$ has length equal to 9680 =(440)(22), and is much shorter as a straight-line program.

\section{Proof of Theorem \ref{neg}}\label{abelian}
Recall that any word $w(x_1,\dots,x_n)$ can be written in the form 
\[
w=x_1^{k_1}\dots x_n^{k_n} v(x_1,\dots,x_n), \text{ where $v\in \textbf{F}_n'.$}
\]

By applying Nielson transformations to $w$, we see that $w$ is automorphic to a word $w'=x_1^k c$ where $k$ is $\textbf{gcd}(k_1,\dots,k_n)$ and $c\in \textbf{F}_n'$. Moreover, $w$ is a law on a group $G$ if and only if $w'$ is a law on $G$. Since automorphic words have the same image over a group $G$, we see that $w(G)=w'(G)$. Hence for a finite abelian group the only images of word maps are exactly the images of the power maps, e.g., $\{x^k:x \in G\}$ for some $k$. We now prove Theorem \ref{neg} showing that not every subset of a group $G$ that is closed under endomorphisms occurs as word map. 

\begin{proof}[\textbf{Proof of Theorem \ref{neg}.}] Let $G=\langle a| a^{12} \rangle $ be the cyclic group of order 12, then the images of the power maps in $G$ are exactly 
\[
1=\{x^{12}\}, G=\{x^1\}, \{1,a^2,a^4,a^6,a^8,a^{10}\}=\{x^2\},\{1,a^3,a^6,a^9\}=\{x^3\}\] \[ \{1,a^4,a^8\}= \{x^4\},\{1,a^6\}=\{x^6\}.
\]

Any union of subsets closed under endomorphisms is closed under endomorphisms. However, there is no power map, equivalently no word map, that has the set $\{1,a^2,a^3,a^4,a^6,a^8,a^9,a^{10}\}$ as its image.
\end{proof}

\bibliographystyle{amsalpha}
\bibliography{ref}

\newcommand{\etalchar}[1]{$^{#1}$}
\providecommand{\bysame}{\leavevmode\hbox to3em{\hrulefill}\thinspace}
\providecommand{\MR}{\relax\ifhmode\unskip\space\fi MR }
\providecommand{\MRhref}[2]{%
  \href{http://www.ams.org/mathscinet-getitem?mr=#1}{#2}
}
\providecommand{\href}[2]{#2}
\begin{thebibliography}{GLO{\etalchar{+}}18}

\bibitem[BCP97]{Magma}
Wieb Bosma, John Cannon, and Catherine Playoust, \emph{The {M}agma algebra
  system. {I}. {T}he user language}, J. Symbolic Comput. \textbf{24} (1997),
  no.~3-4, 235--265, Computational algebra and number theory (London, 1993).
  \MR{MR1484478}

\bibitem[CH18]{CH}
William Cocke and Meng-Che Ho, \emph{On the symmetry of images of word maps in
  groups}, Comm. Algebra \textbf{46} (2018), no.~2, 756--763. \MR{3764894}

\bibitem[CH19]{CH2}
\bysame, \emph{The probability distribution of word maps on finite groups}, J.
  Algebra \textbf{518} (2019), 440--452. \MR{3873947}

\bibitem[GK00]{KG}
Robert~M. Guralnick and William~M. Kantor, \emph{Probabilistic generation of
  finite simple groups}, J. Algebra \textbf{234} (2000), no.~2, 743--792,
  Special issue in honor of Helmut Wielandt. \MR{1800754}

\bibitem[GKP18]{Gordeev}
N.~L. Gordeev, B.~\`E. Kunyavski\u{\i}, and E.~B. Plotkin, \emph{Geometry of
  word equations in simple algebraic groups over special fields}, Uspekhi Mat.
  Nauk \textbf{73} (2018), no.~5(443), 3--52. \MR{3859398}

\bibitem[GLO{\etalchar{+}}18]{GLOST}
Robert~M. Guralnick, Martin~W. Liebeck, E.~A. O'Brien, Aner Shalev, and
  Pham~Huu Tiep, \emph{Surjective word maps and {B}urnside's {$p^aq^b$}
  theorem}, Invent. Math. \textbf{213} (2018), no.~2, 589--695. \MR{3827208}

\bibitem[Hal36]{Hall}
P.~Hall, \emph{The eulerian functions of a group}, The Quarterly Journal of
  Mathematics \textbf{os-7} (1936), no.~1, 134--151.

\bibitem[LOST10]{LOST}
Martin~W. Liebeck, E.~A. O'Brien, Aner Shalev, and Pham~Huu Tiep, \emph{The
  {O}re conjecture}, J. Eur. Math. Soc. (JEMS) \textbf{12} (2010), no.~4,
  939--1008. \MR{2654085}

\bibitem[Lub14]{Lubotzky}
Alexander Lubotzky, \emph{Images of word maps in finite simple groups}, Glasg.
  Math. J. \textbf{56} (2014), no.~2, 465--469. \MR{3187911}

\bibitem[Neu67]{Ne67}
Hanna Neumann, \emph{Varieties of groups}, Springer-Verlag New York, Inc., New
  York, 1967. \MR{0215899}

\end{thebibliography}

\end{document}